\documentclass{amsart}
\usepackage{amssymb,amsmath,latexsym,amscd,amsfonts}
\usepackage{tikz}
\usepackage{graphics}
\usepackage{psfrag}
\usepackage{wasysym}
\usepackage{verbatim}

\usepackage[none,bottom]{draftcopy}

\def\ignore #1 {}

\newtheorem{thm}{Theorem}
\newtheorem{lem}[thm]{Lemma}
\newtheorem{alg}[thm]{Algorithm}
\newtheorem{defn}[thm]{Definition}
\newtheorem{prop}[thm]{Proposition}
\newtheorem{cor}[thm]{Corollary}
\newtheorem{conjecture}[thm]{Conjecture}

\theoremstyle{definition}

\newtheorem{example}[thm]{Example}

\newtheorem*{conj*}{Conjecture}

\def\Z{\mbox{{\bf Z}}}

\def\Z{\mbox{{$\mathbb{Z}$}}}

\def\x{\mbox{{\bf x}}}

\usepackage{hyperref}
\usepackage{dsfont}
\usepackage{amssymb}
\newcommand{\NN}{\mathds{N}}
\newcommand{\ZZ}{\mathds{Z}}
\newcommand{\QQ}{\mathds{Q}}
\newcommand{\RR}{\mathds{R}}

\begin{document}

\title
{Germ order for one-dimensional packings}

\author[A.~Abrams]{Aaron Abrams}
\author[H.~Landau]{Henry Landau}
\author[Z.~Landau]{Zeph Landau}
\author[J.~Pommersheim]{Jamie Pommersheim}
\author[J.~Propp]{James Propp}
\author[A.~Russell]{Alexander Russell}
\date{\today}
\begin{abstract}
Every set of natural numbers
determines a generating function convergent for $q \in (-1,1)$
whose behavior as $q \rightarrow 1^-$ determines a germ.
These germs admit a natural partial ordering
that can be used to compare sets of natural numbers
in a manner that generalizes both cardinality of finite sets
and density of infinite sets.
For any finite set $D$ of positive integers, call a set $S$ ``$D$-avoiding'' 
if no two elements of $S$ differ by an element of $D$.
We study the problem of determining, for fixed $D$,
all $D$-avoiding sets that are maximal in the germ order.
In many cases, we can show that there is exactly one such set.
We apply this to the study of one-dimensional packing problems.
\end{abstract}

\maketitle

\section{Introduction} \label{sec:intro}

This article is concerned with \emph{packing problems} and \emph{distance-avoiding set problems} in $\NN= \{0,1,2,\dots\}$.
A collection of nonempty finite subsets of $\NN$ is called a \emph{packing} if all sets in the collection are disjoint.
We restrict to the case where all sets in the collection are translates
of some fixed finite set $B \subset \NN$ 
(a \emph{packing body})
that has 0 as its smallest element;
we call such a packing a $B$-packing.
The challenge is to find $B$-packings
that cover as much of $\NN$ as possible.
Meanwhile, given some fixed finite set $D$ of positive integers (a set of \emph{forbidden distances}),
a set $S \subseteq \NN$ is called \emph{$D$-avoiding} 
if no two elements of $S$ differ by an element of $D$.
The challenge is to find $D$-avoiding sets that are as large as possible.
In both situations ($B$-packings and $D$-avoiding sets), 
our notion of ``as large as possible'' 
involves the germ-order mentioned in the title,
which refines both cardinality of finite sets and density of infinite sets.

The $B$-packing problem can be seen as a special case of the $D$-avoiding set problem for suitable $D$:
since the equations $b_1+s_1=b_2+s_2$ and $b_1-b_2=s_2-s_1$ are equivalent,
disjointness of the translates $B+s$ ($s \in S$) is equivalent to 
$S$ being $D$-avoiding for $D=\{|b-b'| : b,b'\in B, \: b\ne b'\}$.
(See also basic fact (\ref{product}) in the next section.)
Though we were originally motivated by the packing problem, our results apply to the general $D$-avoiding set problem.

The primary issue is to define what it means for the set $S$ to be ``as big as possible.''
For instance, consider the $D$-avoiding set problem with $D = \{3,5\}$.
Three $D$-avoiding sets are
\[
S = \{0,2,4,6,\dots\}, \quad
S' = \{1,3,5,7,\dots\}, \quad\text{and}\quad
S'' = \{0,1,2,8,9,10,\dots\}.
\]
(The third of these sets is obtained via a greedy algorithm 
for constructing $D$-avoiding sets that considers elements of $\NN$ in increasing order, 
including each element in the set if it does not introduce a forbidden distance with previous elements.
Alternatively, one can describe $S''$ as the lexicographically first infinite $D$-avoiding set.)
With respect to subset-inclusion, all three sets are maximal:
none of them can be augmented
without creating a distance belonging to $D$.
However, using a different partial order, we will find that 
$S$ is ``bigger'' than $S'$ which is in turn ``bigger'' than $S''$.
We define the partial order $\preceq$, called the \emph{germ order}, using the generating function
$$S_q:=  \sum_{n \in S} q^n.$$
\begin{defn}
For subsets $S,S'$ of $\NN$ write $$S' \preceq S$$ 
iff there exists $\epsilon > 0$ such that
$S'_q \leq S_q$ for all $q$ in the interval $(1-\epsilon,1)$.
In this case we say that $S$ \emph{dominates} $S'$
(and that $S'$ is dominated by $S$);
we sometimes say that $S$ is \emph{bigger than} or \emph{better than} $S'$.
\end{defn}  
That is, we compare sizes of sets $S \subseteq \NN$
by examining the germs ``at $1^-$'' of the corresponding generating functions $S_q$.
In particular the three $\{3,5\}$-avoiding sets $S$, $S'$, $S''$ defined above satisfy $S \succeq S' \succeq S''$.

This definition is reminiscent of Abel's method of evaluating divergent series,
in which one assigns to the series $\sum a_n$
the value $\lim_{q\to 1^-} \sum a_nq^n$
if this limit exists.  
Its application to measuring sets of natural numbers is (apparently) new,
but it is likely to hold little novelty for analytic number theorists,
who have long used the philosophically similar
but technically more recondite notion of Dirichlet density
to measure sets of primes.
Our definition also has thematic links to work from 
the earliest days in the study of infinite series.
For instance, Grandi's formula $1-1+1-1+1-1+\dots= 1/2$ corresponds to the fact 
that the germ of $(2\NN)_q$ exceeds the germ of $(2\NN+1)_q$ by $1/2 + O(1-q)$,
while Callet's formula $1+0-1+1+0-1+\dots= 2/3$ corresponds to the fact
that the germ of $(3\NN)_q$ exceeds the germ of $(3\NN+2)_q$ by $2/3 + O(1-q)$.

Our approach resembles the sort of ``tame nonstandard analysis''
in which $\RR$ is replaced by the ordered ring $\RR(x)$ 
where $1/x$ is a formal infinitesimal
(also known as ``the ring of rational functions ordered at infinity'');
our ordering of rational functions corresponds to that of $\RR(x)$
if one identifies $1/x$ with $1-q$.

This paper documents our search for germ-optimal 
solutions to distance-avoiding set problems.  We call a $D$-avoiding set $S$ a (the) \emph{winner} for $D$
if $S$ dominates $S'$ for every $D$-avoiding set $S'$.  A winner for $D$, if it exists, is necessarily unique,
as it is the maximum element of the poset of $D$-avoiding sets ordered by their germs at $1^-$.
It should however be borne in mind that a general poset
can have multiple maximal elements without having a maximum element,
and that some posets (like $\RR$ with its usual ordering)
have no maximal elements at all.

We will prove (Theorem~\ref{thm:avoid-rat}) that for all $D$, any $D$-avoiding set that is germ-maximal 
must be eventually periodic.
Since the germs of the eventually periodic sets are totally ordered,
this implies that there can be at most one germ-maximal $D$-avoiding set.
It follows that if every $D$-avoiding set is dominated by a maximal $D$-avoiding set, then there is a (unique) winner.

In some cases (Sections \ref{sec:examples}, \ref{sec:nonperiodic}) we can both prove existence of winners and give explicit constructions.
We do not know if a winner exists for every $D$, but in Section~\ref{sec:algorithms} we describe some strategies
for constructing winners, one of which shows that 
a sequence of local improvements of a certain kind must converge.  
We do not know that the limit is a winner, however.

\begin{conjecture} \label{conj:avoid-uniq}
For every finite set $D$ of positive integers,
there exists a winner for $D$.
\end{conjecture}

Regarding periodicity, some of the winners we find in Section \ref{sec:examples} are not just 
eventually periodic but actually periodic.  Theorem~\ref{repeatable} gives a sufficient condition for winners 
to be periodic.  In other cases, however, ``boundary effects'' contribute to winners not being periodic from
the start.

 


This contrasts with distance-avoiding set problems in the integers, where there are no boundary effects.  The existing literature on $D$-avoiding problems in dimension 1 is primarily focused on the boundaryless case, and the typical problem is to optimize the density of $D$-avoiding sets.  It is known in this case that for any finite $D$ there is a periodic $D$-avoiding set realizing the optimal density.
See the survey by Liu [Li] for a description of many results in this direction, including the determination of optimal densities for some $D$.

The motivation for our work was the study of disk packings.
It is our hope that the approach taken here will ultimately lead to results
establishing a strong kind of uniqueness
for optimal sphere-packings in dimensions 2, 8, and 24.
(See [Co] for a survey of the recent breakthroughs
in the study of 8- and 24-dimensional sphere-packing.)
We also hope that the germ approach will have relevance
to the study of densest packings in other dimensions.

For alternative approaches to measuring efficiency of packings, see [Ku].
The most sophisticated of these approaches is that of Bowen and Radin [Bo];
their ergodic theory approach has attractive features
(for instance, it works in spaces with nonamenable symmetry groups),
but it does not seem to work so well when the region being packed
is not the entire space.  Packings in $\NN$
could be viewed as special packings of $\RR^{\geq 0}$;
the lack of symmetry makes it hard to apply the constructions of Bowen and Radin.

See also [Be], [Bl], [Ch], and [Ka] for work on measuring sizes of sets
bearing some philosophical similarity to ours.

\section{Basic Facts}

We begin by recording some basic facts about the generating functions $S_q$.

\begin{enumerate}
\item
The boundedness of the coefficients of $S_q$
implies that $S_q$ converges for all complex $q$ inside the unit circle,
though we will only care about $q$ in the interval $(0,1)$.
\item
\label{product}
If the sets $B+s$ ($s \in S$) are disjoint,
then $(B+S)_q = B_q S_q$,
so maximizing the union of the $S$-translates of $B$
(with respect to germ-order)
is equivalent to maximizing the translation-set $S$.
\item
If $S$ is finite, $S_q = |S| + o(1)$, or equivalently,
$S_q \rightarrow |S|$ as $q \rightarrow 1^-$;
if $S$ is infinite, $S_q$ diverges as $q \rightarrow 1^-$.
\item
\label{arith}
If $S = \{a,a+d,a+2d,\dots\}$ with $a \geq 0$ and $d > 0$, then
\[
S_q = \left(\frac{1}{d}\right) \frac{1}{1-q} + \left(\frac{d-1-2a}{2d}\right) + O(1-q).
\]
\item
More generally, if $S$ is infinite with density $\alpha$, then
\[\label{Laurent}S_q = \alpha\, \frac{1}{1-q} + o\left(\frac{1}{1-q}\right).\]
One can prove this by considering the series $S_q - \alpha \frac{1}{1-q}$;
the assumption that $S$ has density $\alpha$
implies that the series is Cesaro summable,
hence Abel summable,
which implies that the series is $o(\frac{1}{1-q})$.

\item
The set $S$ is \emph{eventually periodic} iff there exist $N \in \NN$ and $d \geq 1$
such that for all $n \geq N$, $n \in S$ iff $n+d \in S$.  In this case $S_q := \sum_{n \in S} q^n$
is a rational function of $q$,
and indeed is of the form $P(q)/(1-q^d)$ for some polynomial $P$.
The converse is also true:
if $S_q$ is rational, then the Skolem-Mahler-Lech theorem tells us that
the set of indices $n$ such that the coefficient of $q^n$ in $S_q$ vanishes
is the union of a finite set and a union of finitely many arithmetic progressions,
which implies that $S$ is eventually periodic.
We call such sets $S$ \textbf{rational}.
Note that this usage coincides with the notion of rationality
for subsets of a monoid in automata theory, specialized to the monoid $\NN$.
When $S$ is rational, $S_q$ has a simple pole at 1,
and letting $t = 1-q$ 
we can expand $S_q$ as a Laurent series $\sum_{n \geq -1} a_n t^n$
where $a_{-1}$ is the density of $S$.
\item
If $S$ is not eventually periodic, then as we noted above
$S_q$ is not rational.
This fact, combined with the fact that the coefficients of $S_q$
belong to the finite set $\{0,1\}$,
allows us to apply a classic result of Szeg\H{o} (see e.g.~[TW])
to conclude that $S_q$ is transcendental
and has the unit circle as its natural boundary.
\end{enumerate}

We now turn to some basic facts about the germ order $\preceq$.

\begin{enumerate}
\setcounter{enumi}{7}
\item 
It is easy to see that $\preceq$ is indeed a partial order.
In the case where $S$ and $S'$ are finite,
the germ-ordering refines ordering by cardinality.
If $S$ and $S'$ have finite but nonempty symmetric difference,
then $S_q$ and $S'_q$ differ by a nonzero polynomial,
and since a polynomial can change sign only finitely often,
one of the sets must strictly dominate the other in the germ-ordering.
Likewise, if $S$ and $S'$ are both eventually periodic and distinct,
$S_q$ and $S'_q$ differ by a nonzero rational function,
and once again one set must strictly dominate the other in the germ-ordering.
Thus $\preceq$ is a total order on the rational subsets of $\NN$,
and by \eqref{Laurent} above it refines the preorder given by comparing densities.
\item
Although $\preceq$ is a total ordering for rational subsets of $\NN$, 
the same is not true for unrestricted subsets of $\NN$;
for instance, if $S$ is the set of natural numbers 
whose base ten expansion has an even number of digits
and $S'$ is its complement,
then it can be shown that $S$ and $S'$ are $\preceq$-incomparable.
\item
For every nonempty $S$, 
$(S+1)_q = q S_q \prec S_q$,
so germs are emphatically not translation-invariant.
\end{enumerate}

\section{Examples of periodic winners} \label{sec:examples}

In this section we give a sufficient condition for a winner to exist.  Under this condition the winner is always periodic.  This condition is not necessary, however, as shown by the examples in Section \ref{sec:nonperiodic} of sets $D$ for which there exists a winner that is not periodic.
We show in Section \ref{sec:main} that winners are always eventually periodic.

For $D$ a fixed finite set, we write $\|D\|$ for the largest element of $D$.
In this section we will often find it convenient to
refer to sets $S\subset\NN$ in terms of their indicator functions; 
thus we will often view $S$ as an element of $\{0,1\}^{\NN}$ 
and will correspondingly write $S$ and subsets of $S$ of the form 
$S\cap[a,b]$ as bit strings (of length $b-a+1$ in the latter case).
We denote the concatenation of bit strings $A$ and $B$ (with $A$ finite)
by $AB$. This slight abuse of notation should cause no confusion.

\subsection{Some periodic winners, including the symmetric case} \label{subsec:periodic}

A $D$-avoiding bit string $R$ that is finite and has length greater than $\|D\|$
is called {\bf repeatable} if the concatenation $RR$ is $D$-avoiding. Note that in this case the infinite 
string $RRR\cdots$ is also $D$-avoiding.

\begin{thm}\label{repeatable}
Fix $D$.  If there is an integer
$m>\|D\|$ such that the (germ-)maximal $D$-avoiding string of length $m$ is repeatable, 
then the infinite string obtained by repeating this string is the winner.  In particular the
winner exists and is periodic.
\end{thm}

An example will demonstrate the idea of the proof.  
With $D=\{3,5\}$ and window size $m=8$, 
the best $D$-avoiding string is $10101010$, which is repeatable.
Theorem~\ref{repeatable} asserts that as a consequence of this,
the infinite periodic string $S$ with repetend $10101010$, i.e., the set of even integers, is the winner. 
To justify this, consider any challenger $S'$ and 
compare $S'$ to $S$ in positions 1--8, then in positions 9--16, then 17--24, etc.  
The set $S$ wins (or ties) every time, hence $S$ dominates $S'$.  
Note that in this example choosing $m=6$ doesn't work, as the best $D$-avoiding 
string of length $6$ is $111000$ which is not repeatable.

\smallskip
\noindent
{\bf Proof of Theorem~\ref{repeatable}:}
Let $m$ be as in the statement of the theorem,
and let $S$ be the set
corresponding to the periodic infinite string whose repetend is the germ-maximal $D$-avoiding string of length $m$.
Let $S'$ be the set corresponding to any other $D$-avoiding infinite string.
Then $|S|_q - |S'|_q$ can be written as the sum
$\sum_{i=0}^{\infty} q^{im} (p(q) - p'_i(q))$ 
where $p(q)$ and $p'_i(q)$ ($i \geq 0$) are polynomials of degree at most $m-1$ 
with all coefficients equal to 0 or 1. 
Since there are only
finitely many possibilities for the coefficients
of $p(q) - p'_i(q)$,
and since the coefficients of $p(q)$
form the germ-maximal $D$-avoiding string of length $m$,
there exists a fixed $\epsilon>0$ such that
for all $i$, the polynomial $p(q) - p'_i(q)$
is either identically zero or else
positive on the interval $(1-\epsilon,1)$.
The claim follows.
$\square$

\smallskip

Theorem~\ref{repeatable} applies to many, but definitely not all, sets $D$.  
For instance we call $D$ {\bf symmetric} if there is an integer $k > \|D\|$ such 
that $i\in D$ iff $k-i\in D$.  
(In other words $D$ is symmetric if $-D$ is a translate of $D$, specifically $-D = D-k$.)  
We call the number $k$ the {\bf offset of symmetry} for $D$.
Corollary \ref{cor:periodic} below shows that Theorem~\ref{repeatable} applies to all symmetric sets.
\begin{example}
Consider $D = \{1,3,6,8\}$, which is symmetric with offset of symmetry $k = 9$.  A union of congruence classes modulo 9 is $D$-avoiding if and only if the corresponding vertices form an independent set in the  \emph{circulant graph} shown in Figure \ref{fig:circulant}.  Here the edges join classes which differ by $\pm1(=\!\pm8)$ or $\pm3(=\!\pm6)$.
\begin{figure}
    \centering
    \begin{tikzpicture}
    \foreach \x in {0,...,8}
    {
        \draw (40*\x:2) -- (40+40*\x :2);
        \draw (40*\x:2) -- (120+40*\x :2);
    }    
    \foreach \x in {0,...,8}
        \draw [fill=white] (40*\x:2)  node[anchor=180+40*\x]{$\x$} circle (2pt);
    \end{tikzpicture}
    \caption{The germ-maximal independent set of vertices in this circulant graph is $\{0,2,4\}$.}
    \label{fig:circulant}
\end{figure}
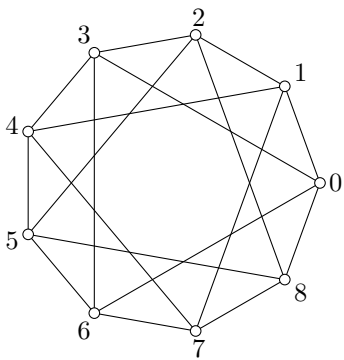
The germ-maximal independent set in this graph is $\{0,2,4\}$, whose corresponding bit string is both optimal among length 9 strings and repeatable, so by Theorem \ref{repeatable} the set $\{n \in \NN: n \equiv 0,2,4 \ (\mbox{mod 9})\}$ is the winner.
\end{example}

\begin{cor} \label{cor:periodic}
If $D$ is symmetric then there is a periodic winner, with period dividing the offset of symmetry.
\end{cor}

\begin{proof}
Fix symmetric $D$ and let $k$ be the offset of symmetry.  We claim
that every $D$-avoiding string of length $k$ is repeatable, so in particular the best
$D$-avoiding string of length $k$ is repeatable.  Thus 
Theorem~\ref{repeatable} applies with $m=k$.

To see why this is true, let $s$ be a $D$-avoiding string of length $k$ and consider
the string $ss$, 
whose two halves correspond to the sets $S$ and $S+k$. 
Suppose this is not $D$-avoiding.  Then there must be elements
$x$ in $S$ and $y$ in $S+k$ with $y-x\in D$.  
By symmetry we have $k-(y-x)=x-(y-k)\in D$.  But $y\in S+k$ so $y-k\in S$, and since $x\in S$, this contradicts that $S$ is $D$-avoiding.
\end{proof}

For example, if $D=\{1,2,\ldots,k-1\}$, then Corollary \ref{cor:periodic} implies that there is a winner, which is easily seen to consist of the multiples of $k$.

\begin{example} \label{ex:nonsymmetric}
For non-symmetric $D$, it is still sometimes possible to apply Theorem~\ref{repeatable}.  
For instance suppose $D=\{1,2,n\}$.  
If $n\not\equiv 0 \mod 3$, then the winner is clearly the periodic sequence with repetend $100$.  
For $D=\{1,2,3n\}$ with $n > 1$, one can show that 
the best $D$-avoiding string of length $3n+1$ is $(100)^n 0$.
This is repeatable, so Theorem~\ref{repeatable} 
tells us we have the winner.
Note however that because $D$ is not symmetric, $D$-avoiding strings $s$ of this (or any) 
length are not guaranteed to be repeatable. So, some 
analysis is required (a) to determine exactly which one is best and (b) to verify that 
it happens to be repeatable.

Some other $D$'s for which we can manually determine that the optimal sequence
(of some length $m$) is repeatable:\\
$D=\{1,3,4\}\ (m=7, \ s=1010000)$ \\
$D=\{2,3,5\}\ (m=7, \ s=1100000)$ \\
$D=\{2,3,6\}\ (m=9, \ s=110001000)$ \\
$D=\{2,3,7\}\ (m=10,\ s=1100011000)$ \\
$D=\{3,4,7\}\ (m=10,\ s=1110000000)$ \\
Interestingly, in each of these cases the optimum is obtained by the greedy algorithm.  
\end{example}

Clearly it would be helpful to have a bound on the window size.
In each example above we have $m\in D+D$, but we do not know if this must always be the case, 
even for sets of the form $D=\{a,b,a+b\}$ corresponding to packings of the body $B=\{0,a,a+b\}$.

In section~\ref{sec:nonperiodic} we show that there exist 3-element $D$'s for which the optimal 
$D$-avoiding set is not periodic (but rather only eventually periodic).
Thus Theorem~\ref{repeatable} cannot be used, and in particular, for such a $D$, 
there can be no (finite) window size $m$ such that the optimal $D$-avoiding string 
of length $m$ is repeatable.

We remark that if the packing body $B$ is symmetric,
then the associated distance set $D$ need not be symmetric;
e.g., consider $B = \{0,1,4,5\}$ and $D = \{1,3,4,5\}$.
Also, the reverse is possible: consider $B = \{0,1,3\}$ (nonsymmetric)
and $D = \{1,2,3\}$ (symmetric).

\section{Structure of winners:  Eventual periodicity} \label{sec:main}

Although we believe (see Conjecture~\ref{conj:avoid-uniq}) that
for every $D$ there is a maximal $D$-avoiding set,
it is a priori conceivable that there is no germ-maximal $D$-avoiding set;
that is, there might exist $D$-avoiding sets
$T_1 \prec T_2 \prec T_3 \prec \dots$
but no $D$-avoiding set that dominates them all.

The closest we have come to proving Conjecture~\ref{conj:avoid-uniq}
is the following result:

\begin{thm} \label{thm:avoid-rat}
For every finite set $D$ of positive integers,
every germ-maximal $D$-avoiding set is rational.
\end{thm}

Note that Example~\ref{example:nonper} shows that the theorem cannot be strengthened to assert that
the germ-maximal sets must be periodic. 

Theorem \ref{thm:avoid-rat}, in combination with the fact that the rational sets are totally ordered under the germ ordering, implies

\begin{cor} \label{cor:uniq}
For every finite set $D$ of positive integers,
there is at most one germ-maximal $D$-avoiding set.
\end{cor}

Our approach to proving Theorem~\ref{thm:avoid-rat} uses a block coding 
of the kind often employed in dynamical systems theory.
We prepare for the proof by introducing this coding
and proving two helpful lemmas.

Let $m = \max(D) + 1$ and replace the indicator sequence of $S$
(an element of $\{0,1\}^\NN$) by a symbolic sequence
using a block code of block length $m$,
with an alphabet containing (at most) $2^m$ symbols, which we will call {\bf letters}.
More concretely, if the indicator sequence of $S$ is written as $(b_0,b_1,b_2,\dots)$
(where $b_n$ is 1 or 0 according to whether $n \in S$ or $n \not\in S$),
then we define the {\bf $m$-block encoding} of $(b_0,b_1,b_2,\dots)$
to be $(w_0,w_1,w_2,\dots)$
where the letter $w_n$ is the $m$-tuple $(b_n,b_{n+1},\dots,b_{n+m-1})$;
we call $w_n$ a {\bf consonant} or a {\bf vowel}
according to whether $b_n = 1$ or $b_n = 0$
(conditions that align with the respective cases $n \in S$ and $n \not \in S$).
Say that a letter $\alpha=(b_1,\dots,b_m)$ in $\{0,1\}^m$ is {\bf legal}
if the set $\{i: b_i=1\}$ is $D$-avoiding;
we let $\mathcal{A}$ be the set of legal letters.
Given two letters $\alpha$ and $\alpha'$ in $\mathcal{A}$,
say that $\alpha'=(b'_1,\dots,b'_m)$ is a {\bf successor} of $\alpha=(b_1,\dots,b_m)$
iff $b'_i = b_{i+1}$ for $1 \leq i \leq m-1$.
For every set $S \subseteq \NN$,
the associated block-encoding $w = (w_0,w_1,w_2,\dots)$
has the property that for all $n \geq 1$,
$w_n$ is a successor of $w_{n-1}$;
$S$ is $D$-avoiding if and only if $w$
has the additional property that every letter $w_n$ is legal.
Call such an infinite word $(w_0,w_1,w_2,\dots)$ {\bf $D$-legal};
there is a one-to-one correspondence between
$D$-avoiding sets and $D$-legal infinite words,
and finding a germ-maximal $D$-avoiding set is equivalent to finding 
a $D$-legal infinite word 
for which the set of locations of consonants is germ-maximal.
We write $w \preceq w'$ iff the associated sets $S,S'$ satisfy $S \preceq S'$.


Let $S$ be a $D$-avoiding subset of $\NN$, and let
$w=(w_0,w_1,\dots)$ be the associated infinite word in $\mathcal{A}^\NN$.

Let $\alpha$ be a letter that occurs infinitely often in the word $w$, and let
$K = \{k \in \NN: w_k = \alpha \} = \{k_0, k_1, k_2, \dots\}$,
where $k_0 < k_1 < k_2 < \cdots$.
After possibly deleting a prefix, we divide the infinite word $w$ into infinitely many subwords
$(w_{k_0},w_{k_0+1},\dots,w_{k_1-1})$,
$(w_{k_1},w_{k_1+1},\dots,w_{k_2-1})$,
$(w_{k_2},w_{k_2+1},\dots,w_{k_3-1})$, \dots.
Each of these finite words 
is associated with the word
$c_k := (w_{k_{i-1}},w_{k_{i-1}+1},\dots,w_{k_{i}-1},w_{k_{i}})$ (for $k \geq 1$)
that both begins and ends with the letter $\alpha$;
define a {\bf circular word} as a word whose first and last letters are the same.
(Note that we are not modding out by cyclic shift of such words.)
Let $\mathcal{C}$ be the set of all circular words beginning and ending with $\alpha$.
We define the {\bf length} of a circular word to be the number of letters it contains,
counting its first and last letter as a single letter.
(Thus, if $\alpha$, $\beta$, and $\gamma$ are letters, 
the circular word $(\alpha,\beta,\gamma,\alpha)$,
which for brevity we may aso write as
$\alpha \beta \gamma \alpha$, has length 3.)
If $c \in \mathcal{C}$ has length $a$ and $c' \in \mathcal{C}$ has length $a'$,
let $c\!:\!c'$ denote the circular word of length $a+a'$ in $\mathcal{C}$
obtained by deleting the final $\alpha$ from $c$ and then concatenating with $c'$.
The operation $:$ is associative, and indeed,
the word $w$ itself can be written as $p\!:\!c_1\!:\!c_2\!:\!c_3\!:\!\dots$,
where $p$ is a possibly empty prefix and the circular words $c_i$ are {\bf primitive}
(i.e., each $c_i$ contains $\alpha$ only at the beginning and at the end).

Every circular word $c \in \mathcal{C}$ is associated 
with a polynomial $P_c = P_c(q)$
(sometimes we will omit the subscript or will write $P_i$ to mean $P_{c_i}$)
whose degree is at most the length $a$ of the circular word $c$
and whose coefficients are 0's and 1's according to whether 
the respective letters in the circular word are vowels or consonants;
we call $P_c$ the {\bf generating function} of $c$.
So if $w = c_1 \!:\! c_2 \!:\! c_3 \!:\! \dots$ is the $D$-legal infinite word
representing the $D$-avoiding set $S$, $S_q$ can be written as 
$P_1 + q^{a_1} P_2 + q^{a_1+a_2} P_3 + \dots = P_1 + A_1 P_2 + A_1 A_2 P_3 + \dots$
where $a_i$ is the length of $c_i$ and $A_i$ is $q^{a_i}$.

For any circular word $c$ with length $a$, we define $|c| := P_c(q) / (1 - q^a)$;
$|c|$ can equivalently be defined as the generating function 
of the infinite periodic word $c\!:\!c\!:\!c\!:\!c\!:\!\dots$.
Given two periodic words $c,c'$ in $\mathcal{C}$ (possibly of different lengths),
write $c \preceq c'$ iff $|c| \preceq |c'|$; 
call this the germ-ordering on circular words.
It is clear from the second way of defining $|c|$
that $|c| = |c\!:\!c| = |c\!:\!c\!:\!c| = \dots$.
We have $|c| = |c'|$ iff $c\!:\!c\!:\!c\!:\!\dots = c'\!:\!c'\!:\!c'\!:\!\dots$.

The following two Lemmas are the linchpins of the proof of Theorem~\ref{thm:avoid-rat}.

\begin{lem} \label{lem:swap}
If $c \preceq c'$, then
$c \preceq c\!:\!c' \preceq c'\!:\!c \preceq c'$.
\end{lem}

\noindent
{\bf Proof:}
Write $|c| = P/(1-A)$ and $|c'| = P'/(1-A')$;
we also have $|c\!:\!c'| = (P + AP') / (1-AA')$
and $|c'\!:\!c| = (P' + A'P) / (1-AA')$.
The stipulated relation $c \preceq c'$ is equivalent to
$P/(1-A) \preceq P'/(1-A')$, or
\begin{equation}
P(1-A') \preceq P'(1-A);
\end{equation}
the desired relations $c \preceq c\!:\!c'$, 
$c\!:\!c' \preceq c'\!:\!c$, and $c'\!:\!c \preceq c'$
are respectively equivalent to
\begin{equation}
P/(1-A) \preceq (P + AP') / (1-AA'),
\end{equation}
\begin{equation}
(P + AP') / (1-AA') \preceq (P' + A'P) / (1-AA'), \ \mbox{and}
\end{equation}
\begin{equation}
(P' + A'P) / (1-AA') \preceq P'/(1-A').
\end{equation}
To prove (2), note that (by cross-multiplying, expanding, and cancelling terms)
we can write it equivalently as $- AA'P \preceq AP' - AP - AAP'$,
which is just (1) multiplied by $A$.
The two denominators in (3) are identical,
so (3) is equivalent to $P + AP' \preceq P' + A'P$,
which in turn is equivalent to (1).
The proof of (4) is similar to the proof of (2).
$\square$

\medskip

Note that the proof also tells us that if $c \prec c'$, 
then $c \prec c\!:\!c' \prec c'\!:\!c \prec c'$.

\begin{lem} \label{lem:lyndon}
If the concatenation $w = c_1\!:\!c_2\!:\!c_3\!:\! \dots$ 
is germ-maximal in the set of $D$-legal words,
then we must have $c_1 \succeq c_2 \succeq c_3 \succeq \dots$ in the germ-ordering.
\end{lem}

\noindent
{\bf Proof:}
We will show that $c_1 \succeq c_2$ 
since that contains the idea of the general argument.
If $c_1 = c_2$ there is nothing to prove, so assume $c_1 \neq c_2$,
and let $w' = c_2\!:\!c_1\!:\!c_3\!:\! \dots$, which must be $D$-legal if $w$ is
(indeed, the whole reason for the block coding was to make this claim true).
The sets $S$ and $S'$ respectively associated with $w$ and $w'$
have finite symmetric difference, so $w$ and $w'$ must be comparable.
Since we are assuming $w$ is germ-maximal,
we must have $w \succeq w'$ in the germ ordering.
That is, we must have $$P_1 + A_1 P_2 \succeq P_2 + A_2 P_1$$
(all the later terms match up and cancel).
But this is equivalent to $P_1 / (1 - A_1) \succeq P_2 / (1 - A_2)$, 
so $c_1 \succeq c_2$ as claimed. $\square$

\medskip

\noindent
{\bf Proof of Theorem~\ref{thm:avoid-rat}:}
Let $S$ be a $D$-avoiding subset of $\NN$
that is germ-maximal.
Let $w=(w_0,w_1,\dots)$ be the associated infinite word in $\mathcal{A}^\NN$.
Suppose first that the letter $w_0 = \alpha$ occurs infinitely often in $w$, and let $c_1,c_2,\ldots$ be the primitive circular words starting and ending with $\alpha$ with $w=c_1\!:\!c_2\!:\!c_3\!:\!\dots$.

By an easy pigeonhole argument, for all $N$ there must exist $i,j \geq N$ with $i < j$
such that the sum of the lengths of the words $c_i$, $c_{i+1}, \dots, c_j$
is a multiple of the length of $c_1$, say $r$ times the length of $c_1$.
Let $w'$ be the word obtained from $w$
by replacing the $j-i+1$ letters $c_i$, $c_{i+1}, \dots, c_j$
by $r$ occurrences of the letter $c_1$.
Let $S$ and $S'$ be the sets associated with $w$ and $w'$, respectively.
Lemma~\ref{lem:lyndon} tells us that $c_1 \succeq c_i \succeq c_{i+1} \succeq \dots \succeq c_j$,
so repeated application of Lemma~\ref{lem:swap} gives
$|c_1\!:\!c_1\!:\!\dots\!:\!c_1| \succeq |c_{i}\!:\!c_{i+1}\!:\!\dots\!:\!c_{j}|$.
If strict inequality holds, then $w' \succ w$, contradicting maximality of $w$.
(Here we use the fact that the difference $S'_q - S_q$
can be expressed as $1-q^n$ times
$|c_1\!:\!c_1\!:\!\dots\!:\!c_1| - |c_{i}\!:\!c_{i+1}\!:\!\dots\!:\!c_{j}|$,
where $n$ is the common value of $ra_1$ and $a_i+a_{i+1}+\dots+a_j$.)
So we must have
$|c_1\!:\!c_1\!:\!\dots\!:\!c_1| = |c_{i}\!:\!c_{i+1}\!:\!\dots\!:\!c_{j}|$,
implying that $c_i,c_{i+1},\dots,c_j$ are all the circular word $c_1$.
Since the circular words $c_i$ are in germ-decreasing order,
this means that $c_1,c_2,\dots,c_N$ are all equal.
Since this is true for all $N$, we must have $w = c_1\!:\!c_1\!:\!\dots c_1$;
that is, $w$ is periodic.

If the letter $w_0$ does not occur infinitely many times in $w$, then we instead find the smallest $i$ for which
the letter $w_i$ does occur infinitely often in $w$.
Ignoring the prefix $w_0,\dots,w_{i-1}$, we may apply the preceding argument
to the letters $w_i,w_{i+1},w_{i+2},\dots$.
The conclusion now is that the word $w$ is eventually periodic, with the periodicity beginning at the letter $w_i$.
$\square$ 

\section{Some non-periodic winners} \label{sec:nonperiodic}

We now demonstrate another technique for proving that
strings are winners, different from the method of section~\ref{sec:examples}.  It is not clear how broadly this approach can be applied.  The following infinite family of examples was discovered by Abrams' student Eric Zhang.



\begin{prop}\label{example:nonper}
Let $D=\{2, 4, 6k+1 \}$ for a positive integer $k$. 
The winner for $D$ is 
\[ S^* = (110000)^k\ 100100\ 100100\ \cdots . \]
\end{prop}

In particular, the winner is not periodic.

\begin{proof}
We split all strings into \emph{bytes} of length 6.
Define two $D$-avoiding strings of $k$ bytes as follows:
\begin{align*}
    A &= (110000)^k \\
    B &= (100100)^k.
\end{align*}

We prove the following facts:
\begin{enumerate}
    \item[Fact 1.] $A$, $AB$ are the germ-maximal $D$-avoiding strings for their respective lengths.
    
    \item[Fact 2.] Let $Q, R$ be strings of length $6k$ such that $QR$ is $D$-avoiding.
    Then either $R\preceq B$ or $QR\preceq BB$.
\end{enumerate}

First, we show $A$ is the germ-maximal $D$-avoiding string of length $6k.$
Let $D'=\{2, 4 \}$.
Note any $D$-avoiding string is also $D'$-avoiding.
For strings of length $6k$, any $D'$-avoiding string is also $D$-avoiding.
The germ-maximal $D'$-avoiding string of length $6k$
is easily seen to be $A$.

To prove $AB$ is the germ-maximal $D$-avoiding string of its length, we can use Fact 2.  Assume Fact 2 is true.
We write $X\prec Y$ to mean $X\preceq Y$ and $X\ne Y$.
Suppose there is a $D$-avoiding string $P$ of length $2\times 6k$ such that $AB \prec P$. 
Let $P=QR$ where $Q,R$ each have length $6k$.
By Fact 1 we know $Q \preceq A$, so if $AB \prec QR$,
we must have $B \prec R.$
Then by Fact 2, we have $P=QR \preceq BB \preceq AB$,
which is a contradiction.

Next we prove Fact 2.
Suppose we fix a $D$-avoiding string $R$ of length $6k$ such that $B\prec R$.
We want to show that $QR \preceq BB$ as long as $QR$ is a $D$-avoiding string of length $2\times6k$.

We first show that the assumption $B\prec R$ restricts the structure of $R$ significantly.
Recall that any $D$-avoiding string is also $D'$-avoiding.
In particular, any byte (of length 6) of any $D$-avoiding string contains at most two 1's.
As each byte of $B$ already has exactly two $1$'s,
$B \prec R$ means the same is true for $R$.
Write $R$ in bytes of length $6$ as $R= R_1R_2\cdots R_k$.
Note that the germ-maximal $D$-avoiding byte is $110000$, and that the runner-up is $100100$.
To satisfy $B \prec R$, therefore,
at least one $R_i$ must equal $110000$.
Further, to avoid distances $2$ and $4$,
we must fill in four consecutive 0's right before $R_i$.
That is, if $R_i=110000$, then $R_{i-1}=110000$.
It follows that $R_1=110000.$

Now, still with the assumption $B\prec R$, consider a $D$-avoiding string $QR$ of length $2\times 6k$.
Suppose $BB \prec QR$.
Write $QR=Q_1\ldots Q_k R_1 \ldots R_k$ in bytes.
By the argument above, $R_1=110000$.
Again to have $BB \prec QR$ each $Q_i$ must have two 1's, so as above working from $Q_k$ back to $Q_1$ we see that each $Q_j=110000$, i.e., $Q=A$.
Yet this is a contradiction,
since the distance between the $1$ in the first position of $Q_1$ and the $1$ in the second position of $R_1$ is the forbidden $6k+1$.
So $QR \preceq BB$ and Fact 2 is proved.

Using Facts 1 and 2, we may now complete the proof.
Let $S$ be any $D$-avoiding infinite string. We will show $S\preceq S^*=ABBBB\cdots$.

Henceforth we write $S=S_1S_2S_3\cdots$ where each $S_i$ is a string of length $6k$, which we refer to as \emph{block} $i$.  (So a block is made of $k$ bytes.)
We use the word \emph{span} to mean a collection of indices of consecutive blocks.

By Fact 1 above, $S^*$ beats or ties $S$ on block $1$ and also on the span consisting of blocks $1$ and $2$.

For each $k$, we inductively construct a partition $P_k$ of the set $\{ 1, 2, \dots , k\}$ into spans such that $S^*$ beats or ties $S$ on each of these spans.
Each span in the partition $P_k$ will have size 1 or 2.
For $k=1, 2$ let $P_k$ be the singletons $\{\{1\}\}$ and $\{\{1, 2\}\}$, respectively.
For $k\geq 3$, assume we have partitions $P_n$ for all $n<k$.
Let $QR=S_{k-1}S_k.$
Recall $S^*$ has $BB$ in the corresponding span.
We define $P_k$ as follows:  
\begin{enumerate}
     \item if $R \preceq B$, let $P_k$ be $P_{k-1}\cup \{\{k\}\}$;
     \item otherwise, by Fact 2, $QR\preceq BB$; in this case let $P_k$ be $P_{k-2}\cup \{\{k-1, k\}\}$.
 \end{enumerate}, 

To show $S\preceq S^*$, it is sufficient to partition the positive integers into spans such that $S^*$ beats or ties $S$ on each span.
We use K\"onig's Lemma to produce such a partition from the sequence of partitions $P_k$. 
Specifically we construct an acyclic graph with vertex set equal to the positive integers by joining each $k\geq 3$ to exactly one of $k-1$ or $k-2$, depending on which partition was used to construct $P_k$.
By K\"onig's Lemma this graph contains an infinite path $\{ j_1, j_2, \cdots\}$, which must be increasing since each vertex only connects to one smaller vertex.
By construction, the partitions $P_{j_1}, P_{j_2}, \cdots$ are nested, so their union is the desired partition of the positive integers.
\end{proof}

A similar analysis can be applied to certain other examples, such as $D=\{4,7,11\}$ and $D=\{5,8,13\}$. Note that these both correspond to packing problems.  In each case the appropriate versions of Facts 1 and 2 are finite checks and can be verified e.g.~using Algorithm~\ref{dynamic} (Section \ref{sec:algorithms}). However $A$ and $B$ can become rather long; for $D=\{5,8,13\}$, for instance, we use blocks of length 18 to prove Fact 2, while Fact 1 requires going a little further out.
We do not know which or how many (non-symmetric) $D$'s, even of size 3, will yield to this approach.

\section{Search for winners:  Some algorithms}\label{sec:algorithms}

\subsection{A finite set of candidates for an optimal sequence} \label{subsec:finite}

We next present a fast algorithm for finding the germ-optimal $D$-avoiding 
bit string of any fixed (finite) length.

We note two properties of our setup:
\begin{itemize}
\item The condition for being a $D$-avoiding string is a local condition, i.e. to check whether the string $S$ is $D$-avoiding can be done locally by checking the condition on contiguous substrings of $S$ of length no more than $\|D\|+1$.  
\item The germ order has the property that given two strings of the same length $A$ and $B$, if $A$ is bigger than $B$ in germ order then $AX$ is bigger than $BX$ for all strings $X$. 
\end{itemize}

These two properties are enough to allow for a standard dynamic programming algorithm on a line to compute, in linear time in the length $l$, a finite list of sequences of length $l$ each of which is optimal conditioned on the values of its final $\|D\|$ bits. The list has one sequence ending with each legal string of length $\|D\|$, so the size of the list is constant in $l$.  The optimal sequence of length $l$ can then be found by comparing the sequences on the list.

\begin{alg}\label{dynamic}
Given a finite set $D\subset \Z_{\geq 0}$, fix any integer $m>\|D\|$. Let $\{\sigma_1, \dots, \sigma_r\}$ denote the set of all $D$-avoiding strings of length $m$.   Then there is an efficient algorithm to compute, for any $k\geq 1$, the set $S_k=\{s_{k,1}, \dots, s_{k,r} \}$, where $s_{k,i}$ is the optimal $D$-avoiding string of length $km$ that ends in the substring $\sigma_i$.  The optimal $D$-avoiding string of length $km$ is then the best element of the finite set $S_k$.
\end{alg}

We briefly spell out the dynamic program.  Initially define $s_{1,i}=\sigma_i$, so that $S_1=\{\sigma_1, \dots, \sigma_r \}$.  Assuming we have  $S_k=\{s_{k,1}, \dots, s_{k,r} \}$ where $s_{k,i}$ is the optimal $D$-avoiding string of length $km$ that ends in the substring $\sigma_i$, we efficiently generate $S_{k+1}$ as follows.  For each $i$, define $s_{k+1,i}$ to be the $D$-avoiding string of biggest germ order from the set $S_{k}\sigma_i$ (that has $r$ elements) consisting of the elements of $S_k$ concatenated with $\sigma _i$.  Then define $S_{k+1}=\{s_{k+1,1}, \dots, s_{k+1,r} \}$.  The optimality of $s_{k+1,i}=\sigma_{j_1}\sigma_{j_2} \dots \sigma _{j_k}\sigma_i$ follows from the guarantee that if $s_{k+1,i}$ is the optimal string of length $(k+1)m$ ending in $\sigma_i$ then the substring $\sigma_{j_1}\sigma_{j_2} \dots \sigma _{j_k}$ must be optimal among strings that end in $\sigma_{j_k}$ and hence was one of the elements considered in $S_k$.

In practice, these lists tend to stabilize fairly quickly.

\subsection{Convergence and maximality of local optimization} \label{subsec:local}

Motivated by the algorithm above, we explore the result of a sequence of ``local improvements,'' each of which replaces a local patch of a string with the optimal substring that is consistent with the adjacent regions. 
Specifically, consider a finite difference set $D$ and two ``boundary'' strings $A, B \in \{0,1\}^{\|D\|}$. In light of the comments above, it follows that for any $\ell \geq \|D\|$ there is a unique maximum string $G \in \{0,1\}^\ell$ for which
\[
 A Z B \preceq A G B 
\]
for all strings $Z$ of length $\ell$. We introduce the notation  $\Gamma_\ell(A,B)$ for this maximum string. It follows that if $w \in \{0,1\}^*$ (or $\{0,1\}^{\NN}$) is a germ-maximal $D$-avoiding string then any appearance of the strings $A$ and $B$ in $w$ separated by exactly $\ell$ symbols must enclose the string $\Gamma_\ell(A,B)$. (Note that it makes sense to define this notion for $\ell < \|D\|$, though in this case one must focus on \emph{consistent} pairs $(\alpha,\beta)$ for which there exists at least one such $x$.) 

In general, for two strings $w, w' \in \{0,1\}^{\NN}$, we write
\[
w \vdash^{A,B}_{\ell} w'
\]
if $w$ can be written $X A Y B Z$ for a string $Y \in \{0,1\}^\ell$ so that $w' = XAGBZ$, where $G = \Gamma_\ell(A,B)$. We likewise define
\[
w \vdash_\ell w'
\]
if $w \vdash_\ell^{A,B} w'$ for some pair $A, B \in \{0,1\}^{\|D\|}$. Observe that
\[
w \vdash_\ell w' \quad \Rightarrow \quad w \preceq w'\,.
\]

\begin{thm} \label{thm:converge} 
Let $D$ be a finite subset of $\NN$ and $\ell \geq \|D\|$. Let $w = w^{(0)} \in \{0,1\}^{\NN}$ be a $D$-avoiding string and let $w^{(1)}, w^{(2)}, \ldots$ be a sequence of elements of $\{0,1\}^{\NN}$ for which
\[
w^{(0)} \vdash_\ell w^{(1)} \vdash_\ell w^{(2)} \vdash_\ell \ldots\,.
\]
Then this sequence converges in the sense that there is a string $w^* \in \{0,1\}^{\NN}$ so that for any position $t$, $w^*_t = w^{(k)}_t$ for all sufficiently large $k$.
\end{thm}

\begin{proof}
Define $g^{(i)}$ to be the power series associated with $w^{(i)}$. Then for each $i$ we may write
\[
g^{(i+1)} = g^{(i)} + X^t p(X)
\]
where $p(X)$ is a polynomial of degree no more than $\ell - 1$ with coefficients in $\{-1, 0, 1\}$. When $w^{(i)} \neq w^{(i+1)}$, the value of $t$ is determined by the length of the common prefix of the two strings.  As $w^{(i)} \preceq w^{(i+1)}$, $X^t p(X) \geq 0$ and hence $p(X) \geq 0$ in the germ order.

Let $P_\ell = \{ a_{\ell-1} X^{\ell-1} + \cdots + a_0 \mid a_i \in \{-1,0,1\} \}$ denote the set of all polynomials of degree at most $\ell-1$ with coefficients in $\{-1,0,1\}$, let
\[
R_\ell = \Bigl\{ x \in \RR \;\Bigm|\; \text{$p(x) = 0$ for some $p(X) \in P_\ell \setminus \{0\}$} \Bigr\}\,,
\]
and define
\[
\epsilon_\ell = \max(\{x \in R_\ell \mid x < 1\})\,.
\]
Observe that if $p(X) \in P_\ell$ exceeds $0$ in the germ order, then $p(x) > 0$ for all $x \in (\epsilon_\ell,1)$. The same can be said for any polynomial of the form $X^t p(X)$, and we conclude that for any point $x_0 \in (\epsilon_\ell,1)$, the values $g^{(i)}(x_0)$ are monotonically increasing. As $g^{(i)}(x_0) \leq 1 + x_0 + x_0^2 + \cdots = 1/(1 - x_0)$, the monotone sequence $g^{(i)}(x_0)$ is bounded and hence converges to a particular value $g^*(x_0)$.

Finally, for a fixed point $x_0 \in (\epsilon_\ell,1)$, define
\[
\epsilon_0 = \min\left(\Bigl\{ |q(x_0)| \,\Bigm|\, q(X) \in P_{\ell} \setminus \{0\}\Bigr\}\right)\,.
\]
Considering two strings $w$ and $\tilde{w}$ for which $w \vdash_\ell \tilde{w}$ corresponding to a substring replacement starting at position $t$, the power series $g$ and $\tilde{g}$ associated with these strings satisfy $\tilde{g}(x_0) = g(x_0) + x_0^t q(x_0)$ for a nonzero polynomial $q(X) \in P_\ell$; hence $\tilde{g}(x_0) \geq g(x) + x_0^t\epsilon_0$.
Then observe that if
\[
\Bigl|g^{(i)}(x_0) - g^*(x_0)\Bigr| <  x_0^t \epsilon_0
\]
for all $i \geq k$ then no such replacement is possible at step $i$ and, indeed, the $t$th bit of all strings $w^{(i)}$ must agree for $i \geq k$. It follows that the sequence $w^{(i)}$ converges pointwise to a particular string $w^*$.
\end{proof}

    Let $D$ be a finite subset of $\NN$ and $\ell \geq \|D\|$. For a $D$-avoiding string $w \in \{0,1\}^*$ and a position $t > \ell$, let $r_t(w)$ be the string obtained by replacing bits $t,t+1,\ldots,t+\ell-1$ with the best possible legal alternative, i.e. $r_t(w)$ is defined by
    $w \vdash^{A,B}_{\ell} r_t(w)$ with $A = w_{t-\|D\|} \ldots w_{t-1}$ and $B = w_{t+\ell} \ldots w_{t+\ell+\|D\|-1}$. 
    
\begin{cor} \label{cor:converge}
    Let $D$ be a finite subset of $\NN$ and $\ell \geq \|D\|$. Let $w \in \{0,1\}^{\NN}$ be a $D$-avoiding string. Let $t_1, t_2, \ldots$ be a sequence of integers so that $t_i > \ell$ for each $i$ and each integer in the set $\{ \ell+1, \ldots\}$ appears infinitely often in the sequence. Then the sequence
        \begin{align*}
        w^{(0)} &= w\,,\\
        w^{(i)} &= r_{t_i}(w^{(i-1)}) \ \mbox{for $i \geq 1$}
        \end{align*}
            converges to an $\ell$-maximal element $w^*$, which is to say that $r_t(w^*) = w^*$ for all $t > \ell$.
\end{cor}

\section{Further thoughts}

\subsection{A topological aside} \label{subsec:topology}

Our germ ordering is not well-behaved relative to
the weak topology on the power set of $\NN$,
wherein a sequence of sets $S_n$ converges if and only if
$S_n \cap F$ is eventually constsant for every finite $F \subset \NN$.
As an illustration of this (related to the famous Ross-Littlewood Paradox),
consider the sequence of sets $S_n = \{n,n+1,\dots,10n\}$;
we have $S_1 \prec S_2 \prec S_3 \prec \dots$,
but it is unclear what the limit of the $S_n$'s should be.
Surely it is not the pointwise limit of the sets, since that is the null set!
One way to understand what is going on here
is to note that, even though for each $n$ there exists $\epsilon_n > 0$
such that $(S_n)_q < (S_{n+1})_q$ for all $q$ in $(1-\epsilon_n,1)$,
we have $\inf \epsilon_n = 0$,
so that the intersection of the intervals $(1-\epsilon_n,1)$ is empty.

This sort of situation comes into play when one tries 
to prove Conjecture~\ref{conj:avoid-uniq} by showing that
$c \succeq c_1,c_2,c_3,\dots$ implies
$c\!:\!c\!:\!c\!:\!\dots \succeq c_1\!:\!c_2\!:\!c_3:\dots$.
If we take $\epsilon_n$ satisfying $|c| \geq |c_n|$ for all $q$ in $(1-\epsilon_n,1)$,
and the infimum of the $\epsilon_n$ is not known to be positive,
then the obvious approach to proving the implication fails.

\subsection{Truncated germs} \label{sec:truncate}

In our work a rational set $S \subseteq \NN$ is replaced by the power series
$\sum_{n \in S} q^n$, which is rewritten as the Laurent series $\sum_{n \geq -1} a_n(1-q)^n$,
and the coefficients $a_{-1},a_0,a_1,a_2,\dots$ are used to put
a total ordering on the rational sets.
The coefficients $a_n$ carry finer and finer information as $n$ increases,
so it is natural to discard this information after some point.
The classical theory of packings retains only $a_{-1}$ (the density of $S$);
we suggest that it is natural to retain both $a_{-1}$ and $a_0$.
That is, we define a non-Archimedean valuation $\nu$
from the set of rational subsets of $\NN$ to $\QQ \times \QQ$,
where we view $\QQ \times \QQ$ as 
the lexicographic product of the ordered ring $\QQ$ with itself.
It can be shown that the pairs $(a_{-1},a_0)$ that occur
are those of the form $(0,k)$ or $(1,-k)$ where $k$ is a nonnegative integer,
along with pairs of the form $(p,q)$
where $p$ is a rational number strictly between 0 and 1
and where $q$ is an arbitrary rational number.
This valuation is not translation-invariant;
if $\nu(S) = (p,q)$, then $\nu(S+1) = (p,q-p)$.
Note that under this valuation, the sets
$\{3,6,9,12,15,18\}$ and $\{1,3,6,9,15,18\}$
discussed at the end of section~\ref{sec:main} have the same size,
since the germs of $\{1\}$ and $\{12\}$ differ by $o(1)$.
The valuation is emphatically not countably additive,
as can for instance be seen by viewing
$\NN$ as a union of singleton sets.

One can try to extend this valuation to various classes of sets
that include but are not limited to the rational subsets of $\NN$.
One way to do this without directly invoking
the expansion of $\sum_{n \in S} q^n$ as a Laurent series in $1-q$
is to define a partial preorder on the power set of $\NN$
(the {\em lim inf preorder})
such that $S$ dominates $S'$ in the lim inf preorder
iff $\liminf_{q \rightarrow 1^-} (\sum_{n \in S} q^n - \sum_{n \in S'} q^n) \geq 0$.
This partial preordering, restricted to the rational sets,
coincides with the total preordering obtained
by factoring the germ-ordering through the valuation $\nu$.

\subsection{Efficiency gaps} \label{sec:gaps}

In the case of packing $\NN$ with translates of
$B = \{0,1,2,\dots,k-1\}$,
there is an appreciable {\em efficiency gap}
between the best packing and all other packings
(where an element $x$ of a non-Archimedean ordered ring extending $\RR$
is said to be {\em appreciable} when there exist positive $r,s$ in $\RR$ with $r<x<s$):

\begin{thm} \label{thm:gap}
For $k \geq 1$ and $D = \{1,2,\dots,k-1\}$,
if $S^*$ is the $D$-avoiding set $\{0,k,2k,3k,\dots\}$
and $S$ is any other $D$-avoiding set,
$S_q \preceq (S^*)_q - \frac{1}{k} + O(1-q)$.
\end{thm}

\noindent
{\bf Proof:} We focus on the case $k=2$ for clarity.
Let $S^* = \{0,2,4,\dots\}$
and let $S$ be some $\{1\}$-avoiding set other than $S^*$.
We can split $S$ into two pieces,
one of which looks like an initial segment of $S^*$
and the other of which doesn't.
In more detail, we write $S$ as the disjoint union of two sets,
one of the form $\{0,2,\dots,2(m-1)\}$ (empty if $m=0$)
and one of the form $\{t_1,t_2,t_3,\dots\}$
(with $t_1 < t_2 < t_3 < \dots$) satisfying
$t_1 \geq 2m+1$, $t_2 \geq 2m+3$, $t_3 \geq 2m+5$, etc.
The germ of $S$ is dominated by the germ of
$\{0,2,\dots,2(m-1)\} \cup \{2m+1,2m+3,2m+5,\dots\}$;
but this germ is the same (up to $O(1-q)$)
as the germ of $\{1,3,5,\dots\}$,
which falls short of the germ of $\{0,2,4,\dots\}$
by $\frac12 + O(1-q)$.
The case $k>2$ is similar. 
$\square$

\medskip

On the other hand, the non-periodic winners of Section \ref{sec:examples}
beat the corresponding periodic contenders by a non-appreciable amount.  This includes the cases $D=\{4,7,11\}$ and $D=\{5,8,13\}$ arising from packing problems. 

\subsection{Choice of regularizer} \label{sec:regularizer}

The germ of $\sum_{n \in S} q^n$ as $q \rightarrow 1^-$ can also be thought of
as the germ of $\sum_{n \in S} e^{-n/s}$ as $s \rightarrow +\infty$;
we think of this as being associated with the function $e^{-t}$
along with a rescaling factor $s$ that measures ``spread''.
The function $e^{-t}$ is a natural 
regularizer to use for packing problems in $\NN$ or $[0,\infty)$;
likewise the function $e^{-t^2}$ would be 
a natural regularizer to use for packing problems in $\ZZ$ or $(-\infty,\infty)$.
Some aspects of the theory are sensitive to the choice of regularizer
but others are not;
e.g., numerical evidence suggests that basic fact (\ref{arith})
from section 2 remains true for the regularizer $q^{n^2}$
(corresponding to the Gaussian kernel $e^{-t^2}$).

\subsection{Connection to sphere-packing} \label{sec:spheres}

Packing problems and distance-avoiding set problems in $\NN$
were chosen as a testbed for ideas about
analogous problems in $\RR^n$,
and more specifically, sphere-packing problems.
Note that the problem of packing spheres of radius 1 in $\RR^n$
is equivalent to the problem of packing points in $\RR^n$
so that no two are at distance less than 2
(the points are the centers of the spheres).
We will not pursue the topic of sphere-packing in depth,
but we will mention the conjectures that motivated this work.

\begin{conjecture}\label{conj:disk-uniq}
Let $S$ be a subset of $\RR^2$,
no two of whose points are at distance less than 2,
and let $S^*$ be the set of center-points in
a hexagonal close-packing of disks of radius 1 in $\RR^2$. 
Let
$$\delta(S) = \liminf_{s \rightarrow \infty} \ \left( 
  \sum_{(x,y) \in S^*} e^{-(x^2+y^2)/s^2} 
- \sum_{(x,y) \in S} e^{-(x^2+y^2)/s^2} \right).$$
Then either $S$ is related to $S^*$ by an isometry of $\RR^2$, 
in which case $\delta(S) = 0$,
or else $S$ is not related to $S^*$ by an isometry of $\RR^2$, 
in which case $\delta(S) > 0$.
\end{conjecture}

\noindent
{\bf Remark:} In private communication, Henry Cohn has shown that
when $S$ is related to $S^*$ by an isometry of $\RR^2$,
$\delta(S)$ is indeed 0.

\medskip

\begin{conjecture}\label{conj:disk-gap}
In the previous Conjecture,
``$\delta(S) > 0$'' can be replaced by ``$\delta(S) \geq 1$'' in the conclusion.
\end{conjecture}
That is, there is an appreciable efficiency-gap for 2-dimensional sphere-packing.

The dichotomy between $\delta(S) = 0$ and $\delta(S) \geq 1$ in Conjecture~\ref{conj:disk-gap}
might at first seem to contradict
the continuity of the summands as a function of the positions of the points;
if all the points move continuously,
won't the lim inf also change continuously?
The catch is that the lim inf can (and often does) diverge.
For instance, if one obtains $S$ from $S^*$
by translating a half-plane's worth of points by $\epsilon > 0$,
or dilating the configuration $S^*$ by a factor of $c > 1$,
then the lim inf diverges, no matter how close $\epsilon$ is to 0,
or how close $c$ is to 1.

Clearly the bound in Conjecture~\ref{conj:disk-gap} cannot be improved,
since removing a single point from $S^*$ gives a set $S$
for which the lim inf is exactly 1.

\bigskip

\noindent
{\sc Acknowledgments:} This work has benefited from conversations with
Tibor Beke, Ilya Chernykh, Henry Cohn, David Feldman, Boris Hasselblatt, 
Alex Iosevich, Sinai Robins, and Omer Tamuz.

\bigskip

\noindent
{\bf References}

\medskip

\noindent
[Be] Vieri Benci, Emanuele Bottazzi, and Maura di Nasso,
``Elementary Numerosity and Measures'',
{\it J.\ Logic and Anal.} {\bf 6} (2014).

\medskip

\noindent
[Bl] Andreas Blass, Mauro Di Nasso, Marco Forti,
``Quasi-selective ultrafilters and asymptotic numerosities'',
{\it Adv. in Math.\ }{\bf 231} (2012), 1462--1486; \\
\href{http://arxiv.org/abs/1011.2089}{{\tt http://arxiv.org/abs/1011.2089}}.

\medskip

\noindent
[Bo] Lewis Bowen and Charles Radin,
``Densest Packing of Equal Spheres in Hyperbolic Space'',
{\it Discrete Comput.\ Geom.\ }{\bf 29} (2003), 23--39.

\medskip

\noindent
[Ch] Ilya Chernykh, ``Non-Trivial Extension of Real Numbers'',
available at \\
\href{http://vixra.org/abs/1701.0617}{\tt http://vixra.org/abs/1701.0617}.

\medskip

\noindent
[Co] Henry Cohn, ``A Conceptual Breakthrough in Sphere Packing'',
{\it Notices of the AMS}, Volume 64, No.\ 2 (February 2017), 102--115.

\medskip

\noindent
[Ka] Fred Katz, ``Sets and Their Sizes'', 
\href{https://arxiv.org/abs/math/0106100}{\tt https://arxiv.org/abs/math/0106100}.

\medskip

\noindent
[Ku] Greg Kuperberg, ``Notions of Denseness'', 
{\it Geom.\ Topol.\ }{\bf 4} (2000), 277--292.

\medskip

\noindent[Li] Daphne Der-Fen Liu,
``From rainbow to the lonely runner: a survey on coloring parameters of distance graphs,''
{\it Taiwanese J.~Math.}, 
Vol.~12 no.~4 (2008), pp.~851--871.

\noindent
[TW] Shiyi Tang and Chunlin Wang, 
``Rational-transcendental dichotomy of power series with a restriction on coefficients'',
{\it Monat.\ f\"ur Math.{}} {\bf 186}(2) (2017), 1--13. 

\end{document}